\documentclass[11pt, a4paper]{amsart}

 \usepackage[ansinew]{inputenc}
 \usepackage[all]{xy}
 \SelectTips{cm}{}
 \usepackage{hyperref}

\usepackage{amsfonts}
\usepackage{amssymb, amsmath,amsthm, mathtools}
\usepackage{tabularx, hyperref}
\usepackage{enumerate}

\usepackage{tikz-cd}

\newtheorem{lemma}{Lemma}[section]
\newtheorem{thm}[lemma]{Theorem}
\newtheorem{prop}[lemma]{Proposition}

\theoremstyle{definition}
\newtheorem{defn}[lemma]{Definition}

\theoremstyle{definition}

\makeatletter
\newcommand\norm{\bBigg@{0.8}}

\makeatother
 
 \newcommand{\indnorm}[2][flex]{\csname #1l\endcsname\|#2%
                                 \csname #1r\endcsname\|\mathclose{}}
                                  \newcommand{\indnorml}[4][flex]{\csname #1l\endcsname\|#2%
                                 \csname #1r\endcsname\|_{#3}^{#4}\mathclose{}}
\newcommand{\sv}[2][flex]{\indnorm[#1]{#2}}

\DeclareMathOperator{\tnorm}{norm}
\newcommand{\nln}[2][flex]{\indnorm[#1]{#2}^{\tnorm}_1}

\newcommand{\N}{\ensuremath {\mathbb{N}}}
\newcommand{\R} {\ensuremath {\mathbb{R}}}

\renewcommand{\rho}{\varrho}
\def\phi{\varphi}

\DeclareMathOperator{\map}{map}
\DeclareMathOperator{\symm}{\textup{symm}}
\DeclareMathOperator{\sgn}{sgn}

\newcommand{\ifsv}[2][norm]{\!\csname #1l\endcsname\bracevert\!#2\!%
                            \csname #1r\endcsname\bracevert\!}
\newcommand{\ifsvlf}[2][norm]{\!\csname #1l\endcsname\bracevert\!#2\!%
                            \csname #1r\endcsname\bracevert\!_{\lf}}

\usepackage{color}
\usepackage{pdfcolmk}

\def\longrightarrow{\rightarrow}

\makeatletter
\@namedef{subjclassname@2020}{%
  \textup{2020} Mathematics Subject Classification}
\makeatother

\begin{document}

\title{Simplicial volume via normalised cycles}

\author[]{Clara L\"{o}h}
\address{Fakult\"{a}t f\"{u}r Mathematik, Universit\"{a}t Regensburg, Regensburg, Germany}
\email{clara.loeh@ur.de}

\author[]{Marco Moraschini}
\address{Fakult\"{a}t f\"{u}r Mathematik, Universit\"{a}t Regensburg, Regensburg, Germany}
\email{marco.moraschini@ur.de}

\thanks{}

\date{\today.\ \copyright{\ C.~L\"oh, M.~Moraschini 2020}.
  This work was supported by the CRC~1085 \emph{Higher Invariants}
  (Universit\"at Regensburg, funded by the DFG)}

\begin{abstract}
  We show that the Connes-Consani semi-norm on singular homology
  with real coefficients, defined via s-modules, coincides with
  the ordinary $\ell^1$-semi-norm on singular homology in all
  dimensions.
\end{abstract}

\maketitle

\section{Introduction}

Connes and Consani introduced a semi-norm on singular homology via
s-modules and established that this semi-norm is equivalent to
the $\ell^1$-semi-norm defined by Gromov~\cite{connesconsani}. Moreover,
they proved that their semi-norm is equal to the $\ell^1$-semi-norm
in the case of surfaces~\cite[Theorem~1.4]{connesconsani}, using
a delicate construction specific to surfaces. 
In this note, we show that the two semi-norms agree in \emph{all}
dimensions, thereby confirming and extending a conjecture of Connes
and Consani~\cite[p.~4]{connesconsani}:

\begin{thm}\label{mainthm}
  Let $X$ be a topological space, let $n \in \N$, let $\alpha \in
  H_n(X;\R)$, and let $\lambda \in \R_{> 0}$. Then $\|\alpha\|_1 <
  \lambda$ if and only if $\alpha$ lies in the image of the canonical
  map~$H_n(X; \| H\R \|_\lambda) \longrightarrow H_n(X;\R)$.
\end{thm}

In particular, the simplicial volume of closed manifolds can also be
expressed in terms of homology of s-modules.

As explained by Connes and Consani, in order to show
Theorem~\ref{mainthm} it suffices to prove that the
$\ell^1$-se\-mi-norm on singular homology can be computed via
\emph{normalised} singular cycles (see Section~\ref{subsec:thmfromprop}):

\begin{prop}\label{mainprop}
  Let $X$ be a topological space and let $n \in \N$. Then, for all~$\alpha \in H_n(X;\R)$,
  we have
  \[ \| \alpha\|_1 = \nln{\alpha}.
  \]
\end{prop}

In Section~\ref{sec:bg}, we recall basic definitions and notation.
The proof of Proposition~\ref{mainprop} is given in
Section~\ref{sec:proof}, based on a symmetrisation
construction.

\section{The (normalised) $\ell^1$-semi-norm}\label{sec:bg}

\subsection{The singular chain complex}

Let $n \in \N$ and let $\Delta^n$ be the standard $n$-simplex.
For~$j \in \, \{0, \cdots, n\}$, we denote by 
$\iota_j^n \colon \Delta^{n-1} \longrightarrow \Delta^n$ 
the affine inclusion of the $j$-th facet of $\Delta^n$. 

Given a topological space $X$, we consider the singular simplicial
set~$S(X)$: For~$n \in \, \mathbb{N}$, we have
$S_n(X) \coloneqq \map(\Delta^n,X)$ and for~$j \in
\{0,\dots, n\}$, the face maps~$\partial_j \colon
S_n(X) \rightarrow S_{n-1}(X)$ are given
by
$$\partial_j(\sigma) \coloneqq \sigma \circ \iota_j^n$$
for all~$\sigma \in \map(\Delta^n,X)$.
The singular chain complex~$C_\bullet(X;\R)$ with
real coefficients is the free $\R$-chain complex associated
with~$S(X)$.

Furthermore, we have the \emph{Moore normalisation~$NC_\bullet(X;\R)$}
of~$C_\bullet(X;\R)$, given by the submodules 
$$
NC_n(X; \R) \coloneqq \bigcap_{j = 0}^{n-1} \ker (\partial_j) \subseteq C_n(X; \R)
$$
and the boundary maps
$
d \coloneqq \partial_n \colon NC_n(X; \R) \longrightarrow NC_{n-1}(X; \R).
$

\begin{defn}[normalised chain]
  A singular chain~$c \in C_n(X; \R)$ is \emph{normalised} if it lies
  in the submodule~$NC_n(X; \R)$.
\end{defn}

\subsection{The $\ell^1$-semi-norms}

We briefly recall Gromov's $\ell^1$-semi-norm on singular
homology~\cite{Grom82}: The $\ell^1$-norm~$\|\cdot\|_1$ on~$C_n(X;\R)$
associated with the basis~$S_n(X)$ induces a semi-norm on~$H_n(X;\R)$,
the \emph{$\ell^1$-semi-norm}, which we will also denote by~$\|\cdot\|_1$.

Following Connes and Consani~\cite{connesconsani}, one can also endow $H_n(X; \R)$ with
the semi-norm induced by the $\ell^1$-norm on the normalised complex $NC_\bullet(X; \R)$:
For~$\alpha \in H_n(X;\R)$ one sets
$$
  \nln{\alpha} \coloneqq
  \inf \bigl\{\sv{c}_1
       \bigm| \text{$c  \in C_n(X; \R)$ is a normalised cycle representing~$\alpha$}
       \bigr\}.
$$
       
Connes and Consani prove that the two semi-norms are
equivalent~\cite[Lemma~3.4]{connesconsani}, namely, for every~$\alpha
\in H_n(X; \R)$, we have
$$
\sv{\alpha}_1 \leq \nln{\alpha} \leq \max(1,2^{n-1}) \cdot \sv{\alpha}_1.
$$
Proposition~\ref{mainprop} states that they are in fact equal.

\subsection{Deriving Theorem~\ref{mainthm} from Proposition~\ref{mainprop}}\label{subsec:thmfromprop}

Connes and Consani introduce a filtration of the s-module~$H\R$ by a
family~$(\| H\R\|_\lambda)_{\lambda \in \R_{>0}}$ of sub-s-modules
and, for topological spaces~$X$, associated singular homology
objects~$(H_n(X;\|H\R\|_\lambda)_{\lambda \in
  \R_{>0}}$~\cite{connesconsani}. Moreover, these come with canonical
maps
\[ \rho_{n,\lambda} \colon H_n(X;\|H\R\|_\lambda) \longrightarrow H_n(X;\R)
\]
to the singular homology
of~$X$~\cite[Section~3.4]{connesconsani}. This filtration defines a
semi-norm on~$H_n(X;\R)$ that is equivalent to
$\|\cdot\|_1$~\cite[Corollary~3.6]{connesconsani}. More precisely, the image
of~$\rho_{n,\lambda}$ coincides with the set of elements~$\alpha
\in H_n(X;\R)$ with~$\nln \alpha <
\lambda$~\cite[Theorem~3.5]{connesconsani}. Therefore,
Theorem~\ref{mainthm} is a direct consequence of
Proposition~\ref{mainprop}.

\section{Proof of Proposition~\ref{mainprop}}\label{sec:proof}

\subsection{Symmetrisation of chains}

We recall the \emph{symmetrisation map} on singular chains, which is
given by averaging singular simplices over all vertex-permutations of
the standard simplex: In the following, let $X$ be a topological space
and $n \in \N$.  Let $\Sigma_{n+1}$ denote the symmetric group
on~$\{0,\dots, n\}$ and $\sgn \colon \Sigma_{n+1} \longrightarrow
\{\pm1\}$ the sign function. 
For a map~$\pi \colon \{0,\dots, k\} \longrightarrow
\{0,\dots, n\}$, we write
$\Delta(\pi) \coloneqq \bigl[\pi(0),\dots, \pi(k)\bigr] \colon \Delta^k \longrightarrow \Delta^n
$ 
for the affine map that extends the map~$\pi$ on the vertices.

\begin{defn}[symmetrisation map]
The \emph{symmetrisation map}
\[
\symm_n \colon C_n(X; \mathbb{R}) \longrightarrow C_n(X; \mathbb{R}) 
\]
is the $\R$-linear map defined on each singular $n$-simplex~$\sigma$ as
\[
\symm_n(\sigma) \coloneqq \frac{1}{(n+1)!} \sum_{\pi \in \, \Sigma_{n+1}}  \sgn(\pi) \cdot
  \sigma \circ \Delta(\pi).
\]
\end{defn}

\begin{lemma}[\protect{\cite[Lemma~2.6]{fujiwaramanning}}]\label{lem:symch}
  The symmetrisation map~$\symm_\bullet$ is a chain
  map~$C_\bullet(X;\R) \longrightarrow C_\bullet(X;\R)$ that is chain
  homotopic to the identity. Moreover, for all~$c \in C_n(X;\R)$, we
  have
  \[
  \sv[big]{\symm_n(c)}_1 \leq \sv{c}_1.
  \]
\end{lemma}

For us, the key observation is that symmetrisation enforces normalisation
on cycles: 

\begin{lemma}[normalisation via symmetrisation]\label{lem:symnorm}
  \hfil
  \begin{enumerate}
  \item
    For all~$j \in \{0,\dots, n\}$, we have
    \[ \partial_j \circ \symm_n = (-1)^j \cdot \partial_0 \circ \symm_n. 
    \]
  \item
    In particular: If $c \in C_n(X;\R)$ is a cycle, then
    $\partial_j(\symm_n(c)) = 0$ for all~$j \in \{0,\dots,n\}$.
  \end{enumerate}
\end{lemma}
\begin{proof}
  \emph{Ad~1.}
  Using the cyclic permutation $\tau_j \coloneqq (j\ j-1\ \dots\ 1\ 0) \in \Sigma_{n+1}$,
  we can re-write $\partial_j \circ \symm_n$ as follows:
  Each permutation~$\pi \in \Sigma_{n+1}$ satisfies
  \begin{align*}
      \Delta(\pi) \circ \iota_j^n
      & = \bigl[\pi(0),\dots,\pi(j-1), \pi(j+1), \dots, \pi(n)\bigr]
      \\
      & = \bigl[\pi\circ\tau_j(1),\dots,\pi\circ\tau_j(j), \pi\circ\tau_j(j+1), \dots, \pi \circ \tau_j(n)\bigr]
      \\
      & = \Delta(\pi \circ \tau_j) \circ \iota_0^n.
  \end{align*}
  Therefore, for all singular $n$-simplices~$\sigma$ on~$X$ we have 
  \begin{align*}
    \partial_j \circ \symm_n(\sigma) 
    &= \frac{1}{(n+1)!} \sum_{\pi \in \, \Sigma_{n+1}}  \sgn(\pi) \cdot
       \sigma \circ \Delta(\pi) \circ \iota_j^n
    \\
    &= (-1)^j \cdot \frac{1}{(n+1)!} \sum_{\pi \in \, \Sigma_{n+1}} \sgn(\pi\circ\tau_j) \cdot
       \sigma \circ \Delta(\pi) \circ \iota_j^n
    \\
    &= (-1)^j \cdot \frac{1}{(n+1)!} \sum_{\pi \in \, \Sigma_{n+1}} \sgn(\pi \circ \tau_j) \cdot
       \sigma \circ \Delta(\pi \circ \tau_j) \circ \iota_0^n
    \\
    &= (-1)^j \cdot \frac{1}{(n+1)!} \sum_{\eta \in \, \Sigma_{n+1}} \sgn(\eta) \cdot
       \sigma \circ \Delta(\eta) \circ \iota_0^n
    \\
    &= (-1)^j \cdot \partial_0 \circ \symm_n(\sigma).
  \end{align*}

  \emph{Ad~2.}
  As $\symm_\bullet$ is a chain map (Lemma~\ref{lem:symch}), 
  if $c \in C_n(X;\R)$ is a cycle, then $\symm_n(c)$ is a cycle, and
  in combination with the first part we see that 
  \begin{align*}
  0 &= \partial \bigl(\symm_n(c)\bigr) 
     = \sum_{j = 0}^n (-1)^j \cdot \partial_j\bigl(\symm_n(c)\bigr) 
     = \sum_{j = 0}^n (-1)^{2j} \cdot \partial_0\bigl(\symm_n(c)\bigr) \\
    &= (n+1) \cdot \partial_0\bigl(\symm_n(c)\bigr).
  \end{align*}
  Therefore, $\partial_0(\symm_n(c)) =0$. Applying the first part once
  more shows that $\partial_j(\symm_n(c)) = 0$ for all~$j \in
  \{0,\dots, n\}$. 
\end{proof}

\subsection{Proof of Proposition~\ref{mainprop}}

We already know that $\sv{\alpha}_1 \leq \nln{\alpha}$ for 
every $\alpha \in H_n(X;\R)$. Let us prove the opposite inequality.
Let $c \in C_n(X;\R)$ be a cycle representing~$\alpha \in H_n(X;\R)$. 
Then, we can consider $\symm_n(c) \in C_n(X;\R)$. By Lemma~\ref{lem:symch}
we know that $\symm_n(c)$ is homologous to $c$ and satisfies 
$\sv{\symm_n(c)}_1 \leq \sv{c}_1$. 
Moreover, Lemma~\ref{lem:symnorm} implies that $\symm_n(c)$ is normalised. 
This shows that 
\[
\nln{\alpha} \leq \sv{\symm_n(c)}_1 \leq \sv{c}_1.
\]
Taking the infimum over all cycles representing~$\alpha$ completes 
the proof.

 
\bibliographystyle{amsalphaabbrv}
\bibliography{svbib}

\providecommand{\bysame}{\leavevmode\hbox to3em{\hrulefill}\thinspace}
\providecommand{\MR}{\relax\ifhmode\unskip\space\fi MR }
\providecommand{\MRhref}[2]{%
  \href{http://www.ams.org/mathscinet-getitem?mr=#1}{#2}
}
\providecommand{\href}[2]{#2}
\begin{thebibliography}{Gro82}

\bibitem[CC20]{connesconsani}
A.~Connes and C.~Consani, \emph{\protect{$\overline{\mathrm{Spec}\; \Z}$} and
  the {G}romov norm}, Theory Appl.\ Categories \textbf{35} (2020), no.~6,
  155--178.

\bibitem[FM11]{fujiwaramanning}
K.~Fujiwara and J.~K. Manning, \emph{Simplicial volume and fillings of
  hyperbolic manifolds}, Algebr.\ Geom.\ Topol. \textbf{11} (2011), 2237--2264.

\bibitem[Gro82]{Grom82}
M.~Gromov, \emph{Volume and bounded cohomology}, Publ.\ Math.\ Inst.\ Hautes
  \'Etudes Sci. \textbf{56} (1982), 5--99.

\end{thebibliography}

\end{document}